\documentclass[12pt]{amsart}
\usepackage{amssymb,latexsym}
\usepackage{enumerate}
\usepackage{graphicx}

\makeatletter
\@namedef{subjclassname@2010}{%
  \textup{2010} Mathematics Subject Classification}
\makeatother

\newtheorem{conjec}{Conjecture}
\newtheorem{teo}{Theorem}[section]
\newtheorem{corol}[teo]{Corollary}
\newtheorem{lema}[teo]{Lemma}
\newtheorem{crite}{Criteria}

\theoremstyle{definition}
\newtheorem{Def}[teo]{Definition}

\numberwithin{equation}{section}

\begin{document}

\baselineskip=17pt

\bigskip

\title[Transcendence of the log-gamma function]{On the possible exceptions for the transcendence of the log-gamma function at rational entries}

\author[F. M. S. Lima]{F.~M.~S.~Lima}
\address{Institute of Physics\\ University of Brasilia\\P.O. Box 04455, 70919-970, Brasilia-DF, Brazil}

\email{fabio@fis.unb.br}


\date{\today}

\begin{abstract}
\quad  In a recent work [JNT \textbf{129}, 2154 (2009)], Gun and co-workers have claimed that the number $\,\log{\Gamma(x)} + \log{\Gamma(1-x)}\,$, $x$ being a rational number between $0$ and $1$, is transcendental with at most \emph{one} possible exception, but the proof presented there in that work is \emph{incorrect}.  Here in this paper, I point out the mistake they committed and I present a theorem that establishes the transcendence of those numbers with at most \emph{two} possible exceptions. As a consequence, I make use of the reflection property of this function to establish a criteria for the transcendence of $\,\log{\pi}$, a number whose irrationality is not proved yet.  This has an interesting consequence for the transcendence of the product $\,\pi \cdot e$, another number whose irrationality remains unproven.
\end{abstract}

\subjclass[2010]{Primary 11J81; Secondary 11J86, 11J91}

\keywords{Log-gamma function, Transcendental numbers}

\maketitle

\section{Introduction}
The gamma function, defined as $\Gamma(x) := \int_{\,0}^{\,\infty} e^{-t} \, {t^{x-1}} \, dt$, $x>0$, has attracted much interest since its introduction by Euler, appearing frequently in both mathematics and natural sciences problems. The transcendental nature of this function at rational values of $x$ in the open interval $(0,1)$, to which we shall restrict our attention hereafter, is enigmatic, just a few special values having their transcendence established. Such special values are: $\Gamma{(\frac12)} = \sqrt{\pi}$, whose transcendence follows from the Lindemann's proof that $\pi$ is transcendental (1882)~\cite{Lindemann}, $\Gamma{(\frac14)}$, as shown by Chudnovsky (1976)~\cite{Chudnovsky}, $\Gamma{(\frac13)}$, as proved by Le Lionnais (1983)~\cite{Lionnais}, and $\Gamma{(\frac16)}$, as can be deduced from a theorem of Schneider (1941) on the transcendence of the beta function at rational entries~\cite{Schneider}.  The most recent result in this line was obtained by Grinspan (2002), who showed that at least two of the numbers $\Gamma{(\frac15)}$, $\Gamma{(\frac25)}$ and $\pi$ are algebraically independent~\cite{Grinspan}.  For other rational values in $(0,1)$ not even irrationality was established for $\,\Gamma{(x)}$.

The function $\,\log{\Gamma{(x)}}$, known as the log-gamma function, on the other hand, received less attention with respect to the transcendence at rational points. In a recent work, however,  Gun, Murty and Rath (GMR) have presented a `theorem' asserting that~\cite{GMR}:
\begin{conjec} \label{conj:GMR}
The number $\log{\Gamma(x)} + \log{\Gamma(1-x)}$ is transcendental for any rational value of $\,x$, $0<x<1$, with at most \textbf{one} possible exception.
\end{conjec}

This has some interesting consequences. For a better discussion of these consequences, let us define a function $f\!: (0,1) \rightarrow \mathbb{R}_{\,+}$ as follows:
\begin{equation}
f(x) := \log{\Gamma(x)} + \log{\Gamma(1-x)} \, .
\label{eq:fx}
\end{equation}
Note that $f(1-x) = f(x)$, which implies that $f(x)$ is symmetric with respect to $x=\frac12$.  By taking into account the well-known \emph{reflection property} of the gamma function
\begin{equation}
\Gamma{(x)} \cdot \Gamma{(1-x)} = \frac{\pi}{\sin{(\pi \, x)}} \, ,
\label{eq:reflex}
\end{equation}
valid for all $x \not \in \mathbb{Z}$, and being $\, \log{\left[\,\Gamma{(x)} \cdot \Gamma{(1-x)}\right]} = \log{\Gamma(x)} + \log{\Gamma(1-x)}$, one easily deduces that
\begin{equation}
f(x) = \log{\left[\frac{\pi}{\sin{(\pi \, x)}}\right]} = \log{\pi} -\log{\sin{(\pi \, x)}} \, .
\label{eq:seno}
\end{equation}
From this logarithmic expression, one promptly deduces that $f(x)$ is differentiable (hence continuous) in the interval $(0,1)$, its derivative being $f\,'(x) = -\,\pi \, / \tan{(\pi x)}$. The symmetry of $f(x)$ around $x=\frac12$ can be taken into account for showing that, being Conjecture~\ref{conj:GMR} true, the only exception would be for $\,x = \frac12$ (see the Appendix). From Eq.~(\ref{eq:seno}), we promptly deduce that $\,\log{\pi} -\log{\sin{(\pi \, x)}}\,$ is transcendental for all rational $x$ in $(0,1)$, the only possible exception being $f(\frac12) = \log{\pi} = 1.1447298858\!\ldots$, which is an interesting number whose irrationality is not yet established.  All these consequences would be impressive, but the proof presented there in Ref.~\cite{GMR} for Conjecture~\ref{conj:GMR} is \emph{incorrect}. This is because those authors implicitly assume that $\,f(x_1) \ne f(x_2)\,$ for every pair of distinct rational numbers $\,x_1, x_2 \in (0,1)$, which is not true, as may be seen in Fig.~\ref{fig:fx}, where the symmetry of $f(x)$ around $x=\frac12$ can be appreciated. To be explicit, let me exhibit a simple counterexample: for the pair $x_1=\frac14$ and $x_2=\frac34$, Eq.~(\ref{eq:seno}) yields $f(x_1) = f(x_2) = \log{\pi} + \log{\sqrt{2}}$ and then  $f(x_1) - f(x_2) = 0$.\footnote{In fact, a null result is found for every pair of rational numbers $\,x_1, x_2 \in (0,1)$ with $x_1 + x_2 = 1$ (i.e., symmetric with respect to $\,x = 1/2$).}  This \emph{null} result clearly makes it invalid their conclusion that $f(x_1) - f(x_2)$ is a \emph{non-null} Baker period.

Here in this short paper, I take Conjecture~\ref{conj:GMR} on the transcendence of $f(x) = \log{\Gamma(x)} + \log{\Gamma(1-x)}$ into account for setting up a theorem establishing that there are at most \emph{two} possible exceptions for the transcendence of $f(x)$, $x$ being a rational in $(0,1)$. This theorem is proved here based upon a careful analysis of the monotonicity of $f(x)$, taking also into account its obvious symmetry with respect to $x=\frac12$. Interestingly, this yields a criteria for the transcendence of $\,\log{\pi}$, an important number in the study of the algebraic nature of special values of a general class of $L$--functions~\cite{GMR2}. Finally, I show that if $\,\,\log{\!(k \, \pi)}\,$ is algebraic for some algebraic $\,k\,$ then $\,\pi \, e$, another number whose irrationality is not proved, has to be transcendental.

\section{Transcendence of $\log{\Gamma{(x)}} + \log{\Gamma{(1-x)}}$ and exceptions}

For simplicity, let us define $\,\mathbb{Q}_{(0,1)}\,$ as $\,\mathbb{Q} \, \bigcap \, (0,1)$, i.e. the set of all rational numbers in the open interval $(0,1)$, which is a countable infinite set.  My theorem on the transcendence of $\,\log{\Gamma{(x)}} + \log{\Gamma{(1-x)}}\,$ depends upon the fundamental theorem of Baker (1966) on the transcendence of linear forms in logarithms, stated below.

\begin{lema}[Baker] \label{lem:Baker}
Let $\alpha_1, \ldots, \alpha_n$ be nonzero algebraic numbers and $\beta_1, \ldots, \beta_n$ be algebraic numbers. Then the number
\begin{eqnarray*}
\beta_1 \, \log{\alpha_1} + \ldots + \beta_n \, \log{\alpha_n}
\end{eqnarray*}
is either zero or transcendental.  The latter case arises if $\, \log{\alpha_1}, \ldots, \log{\alpha_n}$ are linearly independent over $\,\mathbb{Q}$ and $\beta_1, \ldots, \beta_n$ are not all zero.
\end{lema}

\begin{proof} $\,$~See theorems 2.1 and 2.2 of Ref.~\cite{BakerBook}. 
\end{proof}

Now, let us define a \emph{Baker period} according to Refs.~\cite{Sarada,Zagier}.

\begin{Def}[Baker period] \label{def:BakerPeriod}
A Baker period is any linear combination in the form $\beta_1 \, \log{\alpha_1} + \ldots + \beta_n \, \log{\alpha_n}$, with $\alpha_1, \ldots, \alpha_n$ nonzero algebraic numbers and $\beta_1, \ldots, \beta_n$ algebraic numbers.
\end{Def}

From Baker's theorem, it follows that

\begin{corol} \label{cor:naonulo}
Any \emph{non-null} Baker period is a transcendental number.
\end{corol}

Now, let us demonstrate the following theorem, which comprises the main result of this paper.

\begin{teo}[Main result] \label{teo:meu}
The number $\,\log{\Gamma(x)} + \log{\Gamma(1-x)}\,$ is transcendental for all $\,x \in \mathbb{Q}_{(0,1)}$, with at most \emph{two} possible exceptions.
\end{teo}

\begin{proof}
$\,$~Let $f(x)$ be the function defined in Eq.~(\ref{eq:fx}). From Eq.~(\ref{eq:seno}), $f(x) = \log{\pi} -\log{\sin{(\pi \, x)}}$ for all real $x \in (0,1)$.  Let us divide the open interval $(0,1)$ into two adjacent subintervals by doing $(0,1) \equiv (0,\frac12] \, \bigcup \, [\frac12,1)$. Note that $\sin{(\pi \, x)}$ --- and thus $f(x)$ --- is either a monotonically increasing or decreasing function in each subinterval.  Now, suppose that $f(x_1)$ and $f(x_2)$ are both algebraic numbers, for some pair of distinct real numbers $x_1$ and $x_2$ in $(0,\frac12]$. Then, the difference
\begin{equation}
f(x_2) - f(x_1) = \log{\sin{(\pi \, x_1)}} - \log{\sin{(\pi \, x_2)}}
\end{equation}
will, itself, be an algebraic number.  However, as the sine of any rational multiple of $\pi$ is an algebraic number~\cite{Niven,Dresden}, then Lemma~\ref{lem:Baker} guarantees that, being $x_1,x_2 \in \mathbb{Q}$, then $\log{\sin{(\pi \, x_1)}} - \log{\sin{(\pi \, x_2)}}$ is either null or transcendental. Since $\sin{(\pi x)}$ is a continuous, monotonically increasing function in $(0,\frac12)$, then $\sin{\pi x_1} \ne \sin{\pi x_2}$ for all $x_1 \ne x_2$ in $(0,\frac12]$. Therefore, $\log{\sin{(\pi \, x_1)}} \ne \log{\sin{(\pi \, x_2)}}$ and then $\log{\sin{(\pi \, x_1)}} - \log{\sin{(\pi \, x_2)}}$ is a \emph{non-null} Baker period. From Corol.~\ref{cor:naonulo}, we know that non-null Baker periods are transcendental numbers, which contradicts our initial assumption. Then, there is at most one exception for the transcendence of $f(x)$, $x \in \mathbb{Q} \, \bigcap \, (0,\frac12]$.  Clearly, as $\sin{(\pi x)}$ is a continuous and monotonically decreasing function for $x \in [\frac12,1)$, an analogue assertion applies to this complementary subinterval, which yields another possible exception for the transcendence of $f(x)$, $x \in \mathbb{Q} \, \bigcap \, [\frac12,1)$.
\end{proof}

It is most likely that not even an exception takes place for the transcendence of $\log{\Gamma{(x)}} + \log{\Gamma{(1-x)}}$ with $x \in \mathbb{Q}_{(0,1)}$. If this is true, then the number $\,f(\frac12) = \log{\pi}\,$ would be transcendental. If there are exceptions, however, then their quantity --- either one or two, according to Theorem~\ref{teo:meu} --- will determine the transcendence of $\,\log{\pi}$. The next theorem summarizes these connections between the existence of exceptions to the transcendence of $f(x)$, $x \in \mathbb{Q}_{(0,1)}$, and the transcendence of $\,\log{\pi}$.

\begin{teo}[Exceptions] \label{teo:excecoes}
With respect to the possible exceptions to the transcendence of $\,\log{\Gamma{(x)}} + \log{\Gamma{(1-x)}}$, $x \in \mathbb{Q}_{(0,1)}$, exactly one of the following statements is true:
\begin{itemize}
\item[(i)] There are no exceptions, hence $\,\log{\pi}$ is a transcendental number;

\item[(ii)] There is only one exception and it has to be for $x=\frac12$, hence $\,\log{\pi}$ is an algebraic number;

\item[(iii)] There are exactly two exceptions for some $x \ne \frac12$, hence $\,\log{\pi}$ is a transcendental number.
\end{itemize}
\end{teo}

\begin{proof}
$\,$~If $f(x) = \log{\Gamma{(x)}} + \log{\Gamma{(1-x)}}$ is a transcendental number for every $x \in \mathbb{Q}_{(0,1)}$, item(i), it suffices to put $x=\frac12$ in Eq.~(\ref{eq:seno}) for finding that $f(\frac12) = \log{\pi}$ is transcendental.  If there is \emph{exactly one} exception, item (ii), then it has to take place for $x = \frac12$, otherwise (i.e., for $x \ne \frac12$) the symmetry property $f(1-x)=f(x)$ would yield algebraic values for \emph{two} distinct values of the argument. Therefore, $f(\frac12)=\log{\pi}$ is the only (algebraic) exception in this case. If there are two exceptions, item (iii), then they have to be symmetric with respect to $x =\frac12$, otherwise, by the property $f(1-x)=f(x)$, we would find more than two exceptions, which is prohibited by Theorem~\ref{teo:meu}. Indeed, if one of the two exceptions is for $x=\frac12$, then the other, for $x \ne \frac12$, would yield a third exception, corresponding to $\,1-x \ne \frac12$, which is again prohibited by Theorem~\ref{teo:meu}. Then the two exceptions are for values of the argument distinct from $\frac12$ and then $f(\frac12) = \log{\pi}$ is a transcendental number. 
\end{proof}

From this theorem, it is straightforward to conclude that

\begin{crite}[transcendence of $\,\log{\pi}$]
\label{criterio}
The number $\,\log{\pi}$ is algebraic if and only if $\,\log{\Gamma{(x)}} + \log{\Gamma{(1-x)}}$ is a transcendental number for every $x \in \mathbb{Q}_{(0,1)}$, except $x = \frac12$.
\end{crite}


An interesting consequence of Criteria~\ref{criterio}, together the famous Hermite-Lindemann (HL) theorem, is that if the number $\,\log{\Gamma{(x)}} + \log{\Gamma{(1-x)}}\,$ is algebraic for some $\,x \in \mathbb{Q}_{(0,1)}\,$ then the number $\,\pi \cdot e = 8.5397342226\!\ldots$, another number for which not even an irrationality proof is known, has to be transcendental. Let me proof this assertion based upon a logarithmic version of the HL theorem.

\begin{lema}[HL] \label{lem:HL}
For any non-zero complex number $w$, one at least of the two numbers $w$ and $\exp{(w)}$ is transcendental.
\end{lema}
\begin{proof}
$\,$~See Ref.~\cite{Baker2} and references therein.
\end{proof}

\begin{lema}[HL, logarithmic version] \label{lem:HL1}
For any positive real number $z$, $z \ne 1$, one at least of the real numbers $z$ and $\log{z}$ is transcendental.
\end{lema}

\begin{proof}
$\,$~It is enough to put $w=\log{z}$, $z$ being a non-negative real number, in Lemma~\ref{lem:HL} and to exclude the singularity of $\log{z}$ at $z=0$.
\end{proof}

\begin{teo}[Transcendence of $\,\pi \, e$]
If the number $\,\log{\Gamma{(y)}} + \log{\Gamma{(1-y)}}\,$ is algebraic for some $\,y \in \mathbb{Q}_{(0,1)}$, then the number $\,\pi \, e\,$ is transcendental.
\end{teo}

\begin{proof}
$\,$~Let us denote by $\,\overline{\mathbb{Q}}\,$ the set of all algebraic numbers and $\overline{\mathbb{Q}}^{\,*}$ the set of all non-null algebraic numbers. First, note that $k(y) := {\,1/\sin(\pi y)} \, \in \overline{\mathbb{Q}}^{\,*}$ for every $y \in \mathbb{Q}_{(0,1)}$ and that, from Eq.~(\ref{eq:seno}), $\,\log{\Gamma{(y)}} + \log{\Gamma{(1-y)}} = \log{[k(y) \, \pi]}$.  Now, being $\,\log{[k(y) \, \pi]} \in \overline{\mathbb{Q}}\,$ for some $y \in \mathbb{Q}_{(0,1)}$, then $\,1 + \log{[k(y) \, \pi]}\,$ would also be an algebraic number. Therefore, $\log{e} + \log{[k(y) \, \pi]} = \log{[k(y) \, \pi \, e]} \in \overline{\mathbb{Q}}$ and, by Lemma~\ref{lem:HL1}, the number $k(y) \, \pi \, e$ would be either transcendental or $1$. However, it cannot be equal to $1$ because this would imply that $k(y) = 1/(\pi \, e) < 1$, which is impossible since $\,0 < \sin{(\pi \, y)} \le 1$, $\forall \, y \in (0,1)$. Therefore, the product $\,k(y) \, \pi \, e\,$ has to be a transcendental number. Since $k(y) \in \overline{\mathbb{Q}}^{\,*}$, then $\,\pi \, e\,$ has to be transcendental.
\end{proof}



\section*{Appendix}

Let us explain why Conjecture~\ref{conj:GMR} --- i.e., the assertion that $\,f(x)=\log{\Gamma{(x)}}+\log{\Gamma{(1-x)}}\,$ is transcendental with at most \emph{one} possible exception, $x$ being a rational in $(0,1)$ --- implies that if an exception exists then it has to be just $\,f(\frac12) = \log{\pi}$.  The fact that $f(1-x)=f(x)$ for all $x \in (0,1)$ implies that, if the only exception would take place for some rational $\,x \ne \frac12\,$, then automatically there would be another rational $\,1-x \ne \frac12\,$ at which the function would also assume an algebraic value, contrarily to Conjecture~\ref{conj:GMR}.

\newpage

\section*{Figures}

\begin{figure}[h]
\centering
\includegraphics[scale=0.39]{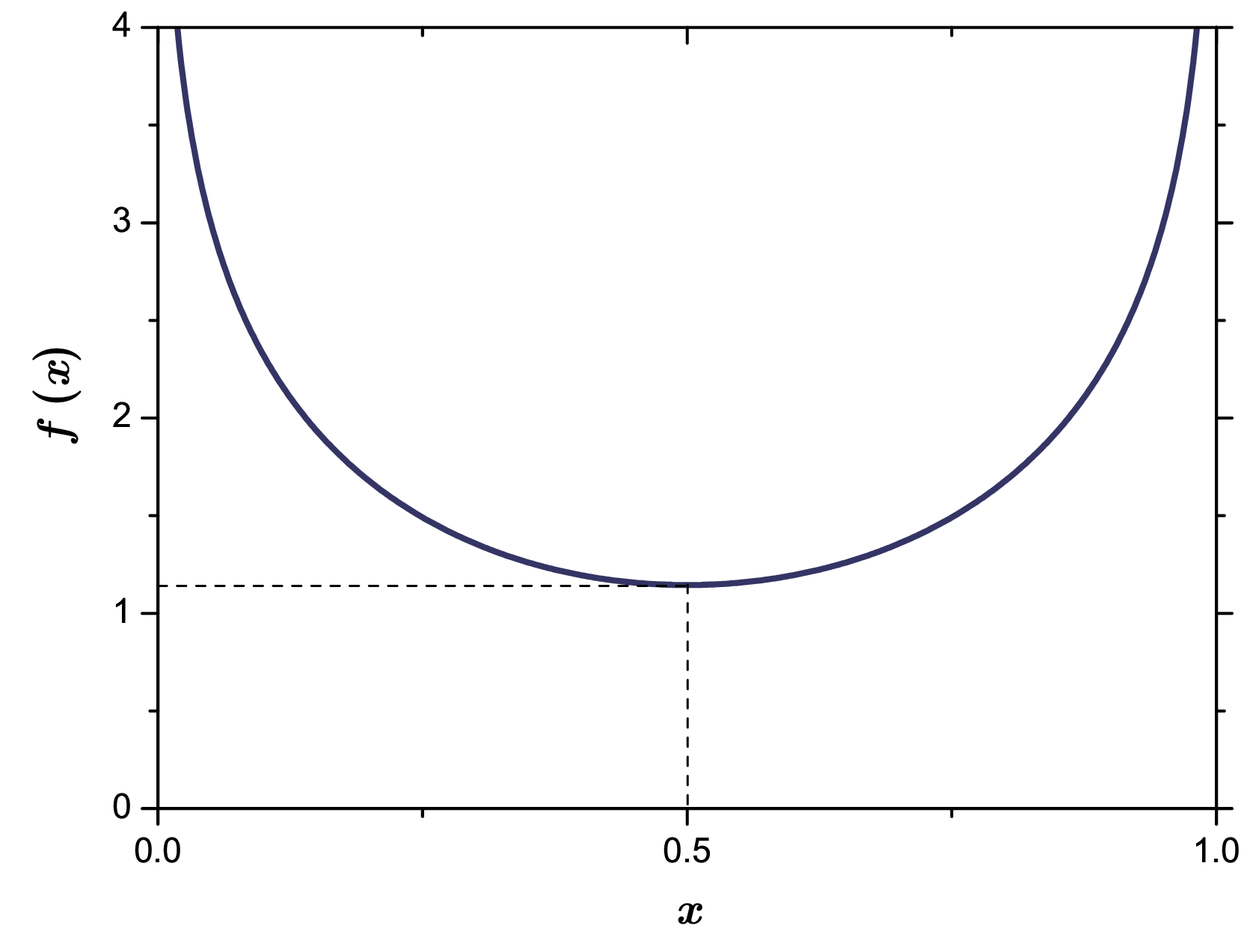}
\caption{The graph of the function $f(x) = \log{\Gamma(x)} + \log{\Gamma(1-x)} = \log{\pi} -\log{\left[\sin{(\pi x)}\right]}$ in the interval $(0,1)$. Since $\,f(1-x) = f(x)$, the graph is symmetric with respect to $x=\frac12$. Note that, as $0 < \sin{(\pi \, x)} \le 1$ for all $x \in (0,1)$, then $\log{\sin{\!(\pi x)}}\le 0$, and then $f(x) \ge \log{\pi}$ and the minimum of $f(x)$, $x \in (0,1)$, is attained at $\,x=\frac12$, where $f(x)$ evaluates to $\log{\pi}$. The dashed lines highlight this point.}
\label{fig:fx}
\end{figure}

\end{document}